\documentclass{amsart}

 \usepackage{amssymb}
\usepackage{amsthm}
\usepackage{fullpage}
\newtheorem{defi}{Definition}[section]
\newtheorem{theorem}{Theorem}[section]
\newtheorem{lemma}{Lemma}[section]
\newtheorem{remark}{Remark}[section]
\numberwithin{equation}{section}
\newtheorem*{theorem*}{Theorem}

\def\O{\Omega}
\def\R{{\mathbb R}}

\def\d{\delta}

\def\d{\mbox{diam}}

\begin{document}


\title[Weighted fractional Sobolev spaces]{Weighted fractional Sobolev spaces as interpolation spaces in bounded domains}

\author{Gabriel Acosta}
\address{IMAS (UBA-CONICET) and Departamento de Matem\'atica, Facultad de Ciencias Exactas y Naturales, Universidad de Buenos Aires, Ciudad Universitaria, 1428 Buenos Aires, Argentina}
\email{gacosta@dm.uba.ar}

\author{Irene Drelichman}
\address{IMAS (UBA-CONICET), Facultad de Ciencias Exactas y Naturales, Universidad de Buenos Aires, Ciudad Universitaria, 1428 Buenos Aires, Argentina}
\email{irene@drelichman.com}

\author{Ricardo G. Dur\'an}
\address{IMAS (UBA-CONICET) and Departamento de Matem\'atica, Facultad de Ciencias Exactas y Naturales, Universidad de Buenos Aires, Ciudad Universitaria, 1428 Buenos Aires, Argentina}
\email{rduran@dm.uba.ar}

\thanks{Supported by FONCYT under grants  PICT-2018-03017 and PICT-2018-00583, and by Universidad de Buenos Aires under grant 20020160100144BA}

\begin{abstract}
We characterize the real interpolation space between a weighted $L^p$ space and a weighted Sobolev space in arbitrary bounded domains in $\mathbb{R}^n$, with weights that are positive powers of the distance to the boundary. 
\end{abstract}


\maketitle

\section{Introduction}
The aim of this article is to contribute to the study of weighted fractional Sobolev spaces in arbitrary bounded domains in $\mathbb{R}^n$, when the weights are positive powers of the distance to the boundary, by characterizing them as real interpolation spaces. The results we obtain are new even in the case of smooth domains. 

To be more precise, let us introduce some notation. For any bounded domain  $\Omega\subset \R^n$ we denote by  $d(x)=d(x, \partial\Omega)$  the distance from $x$ to the boundary of $\Omega$. We will consider the weighted Sobolev spaces 
$$
W^{1,p}(\Omega, d^\alpha, d^\beta)= \{f \in L^p(\Omega, d^\alpha) : \|\nabla f\|_{L^p(\Omega, d^{\beta})}< \infty\},
$$
where $\alpha, \beta \ge 0$ and $\|f\|_{L^p(\Omega, d^{\alpha})}=\|f d^\frac{\alpha}{p}\|_{L^p(\Omega)}$,  and their fractional counterpart
$$
\widetilde W^{s,p}(\Omega, d^\alpha, d^\beta)= \{f \in L^p(\Omega, d^\alpha) : | f |_{\widetilde W^{s,p}(\Omega, d^{\beta})}< \infty\},
$$
where $0<s<1$ and
\begin{equation}
\label{eq:seminorma_frac}
|f|_{\widetilde W^{s,p}(\Omega, d^\beta)}^p = \int_\Omega \int_{|x-y|<\frac{d(x)}{2}} \frac{|f(x)-f(y)|^p}{|x-y|^{n+sp}} \,  dy  \, d(x)^\beta \,dx .
\end{equation}
In the unweighted case, we will simply write $W^{1,p}(\Omega) = W^{1,p}(\Omega, 1, 1)$, $\widetilde W^{s,p}(\Omega) = \widetilde W^{s,p}(\Omega, 1,1)$, and ${|\cdot|_{\widetilde W^{s,p}(\Omega)}}= |\cdot|_{\widetilde W^{s,p}(\Omega,1)}$.

Observe that, due to the fact that the region of integration of its inner integral is restricted to  $|x-y|<\frac{d(x)}{2}$, the seminorm  \eqref{eq:seminorma_frac}   is equivalent to
\begin{equation}
\label{eq:seminorma_frac_sim}
 \int_\Omega \int_{|x-y|<\frac{d(x)}{2}} \frac{|f(x)-f(y)|^p}{|x-y|^{n+sp}} \, \delta^{\beta}(x,y) \, dy \,dx,
\end{equation}
where $\delta(x,y)=\min\{d(x), d(y)\}$. Moreover, replacing $\frac12 d(x)$ by $\tau d(x)$ for any fixed $\tau \in (0,1)$ in   
\eqref{eq:seminorma_frac} and \eqref{eq:seminorma_frac_sim},  the norms of the associated spaces are equivalent (see Remark \ref{remark22}).

The seminorm $|f|_{\widetilde W^{s,p}(\Omega)}$ was introduced in the context of Poincar\'e and Sobolev-Poincar\'e inequalities in John domains in \cite{HV} and further results on its relevance for these inequalities were obtained in \cite{DD-fracpoincare, DIV, G, MP}. It is equivalent to the usual Gagliardo seminorm in $\Omega \times \Omega$ in Lipschitz domains (see  \cite[equation (13)]{D}) and, more generally, in uniform domains (see \cite[Corollary 4.5]{PS}). More importantly,  it is a replacement for  the non-inclusion $W^{1,p}(\Omega) \not \subset W^{s,p}(\Omega)$ when $\Omega$ is an irregular domain (see \cite[Example 2.1]{DD-interpolacion} for such an example). In fact, the embedding  
$W^{1,p}(\Omega)  \subset \widetilde W^{s,p}(\Omega)$ holds for any bounded domain (see \cite[Lemma 2.2]{DD-BBM}) and one also has the Bourgain-Br\'ezis-Mironescu limit property, namely that for any $1<p<\infty$ and $f\in L^p(\Omega)$,  $\lim_{s\to 1^-} (1-s) |f|_{\widetilde W^{s,p}(\Omega)}^p = K_{n,p} \|\nabla f\|_{L^p(\Omega)}^p$, where $K_{n,p}$ is an explicit constant and the right-hand side of the inequality is understood to be infinity if $f\not\in W^{1,p}(\Omega)$ (see \cite[Theorem 1.1]{DD-BBM}). A deeper result in \cite[Theorem 3.2]{DD-interpolacion} is that for a general class of irregular domains one has  $(L^p(\Omega), W^{1,p}(\Omega))_{s,p} = \widetilde W^{s,p}(\Omega)$ which is a generalization of the classical result $(L^p(\Omega), W^{1,p}(\Omega))_{s,p} =  W^{s,p}(\Omega)$ for smooth domains (see \cite[Section 1.1.2, Exercise 7]{L} or \cite[Lemmas 35.2 and 36.1]{T}). Here, $(L^p(\Omega), W^{1,p}(\Omega))_{s,p}$ stands for the real interpolation space (see Section 2 for its definition).

Perhaps suprinsingly, while for irregular domains it may happen that $\widetilde W^{s,p}(\Omega) \neq W^{s,p}(\Omega)$, it turns out that,  for every $\alpha \ge 0$,  $\widetilde W^{s,p}(\Omega, d^{\alpha p}, d^{(\alpha+s)p})=  W^{s,p}(\Omega, d^{\alpha p}, \delta^{(\alpha+s)p})$  with equivalence of norms, where

$$
 W^{s,p}(\Omega, d^{\alpha p}, \delta^{(\alpha+s)p})= \{f \in L^p(\Omega, d^{\alpha p}) : | f |_{ W^{s,p}(\Omega, \delta^{(\alpha+s)p})}< \infty\},
$$
and 
 \begin{equation}
|f|_{W^{s,p}(\Omega, \delta^\beta)}^p = \int_\Omega \int_\Omega \frac{|f(x)-f(y)|^p}{|x-y|^{n+sp}}  \delta(x,y)^\beta \,  dy  \,  \,dx
\end{equation}
(see Lemma \ref{equivalentes}).

As previously announced, our aim is to  characterize the above fractional weighted spaces as  real interpolation spaces. More precisely, we prove:

\begin{theorem}
\label{inter_general}
Let $\Omega\subset \R^n$ be an open bounded domain,  $0<s<1$ and $\alpha \ge 0$, then 
we have
$$(L^p(\Omega, d^{\alpha p}), W^{1,p}(\Omega, d^{\alpha p}, d^{(\alpha+1)p}))_{s,p}= \widetilde W^{s,p}(\Omega, d^{\alpha p}, d^{(\alpha+s)p})=  W^{s,p}(\Omega, d^{\alpha p}, \delta^{(\alpha+s)p})$$
with equivalence of norms.
\end{theorem}

 The weighted space $ W^{s,2}(\Omega, 1, \delta^{\beta})$ was used in Lipschitz domains in \cite{AB} to establish regularity results and to develop efficient numerical methods for nonlocal equations involving the fractional Laplace operator, while the  spaces $\widetilde W^{s,p}(\Omega, d^\alpha, d^\beta)$ have  been proved useful for Poincar\'e and Sobolev-Poincar\'e inequalities in John domains \cite{DD-fracpoincare, MP}. But, to the best of our knowledge, this is the first result where they are studied in the context of interpolation spaces. Related interpolation results, with applications to elliptic eigenvalue problems, were studied in \cite{P} but for Lipschitz domains only, and with somewhat different characterizations.  
 
The rest of the paper is as follows. We will first prove the equivalence $\widetilde W^{s,p}(\Omega, d^{\alpha p}, d^{(\alpha+s)p})= W^{s,p}(\Omega, d^{\alpha p}, \delta^{(\alpha+s)p})$ and quickly review some necessary preliminaries and notations. Then, for simplicity, we will write the proof of the above theorem in full detail  using the norm of the space $\widetilde W^{s,p}(\Omega, 1, d^{sp})$ for the special case $\alpha=0$,  as all the ideas and technical difficulties are already present there. Finally, in the last section of this article we will indicate how the proof  for $\alpha=0$ can be modified to obtain the proof of Theorem \ref{inter_general} for any $\alpha> 0$.

\section{Notation and Preliminary Results}

Throughout this article, $1\le p<\infty$ and $p^\prime$ is its conjugate exponent, $\frac{1}{p}+\frac{1}{p^\prime}=1$, while $C$ represents a positive constant that might change from line to line.

\begin{lemma}
\label{equivalentes}
Let $\Omega\subset \R^n$ be an open bounded domain,  $0<s<1$ and $\alpha \ge 0$. Then,  
$$\widetilde W^{s,p}(\Omega, d^{\alpha p}, d^{(\alpha+s)p})=  W^{s,p}(\Omega, d^{\alpha p}, \delta^{(\alpha+s)p})$$
  with equivalence of norms.
\end{lemma}
\begin{proof}
Clearly, $W^{s,p}(\Omega, d^{\alpha p}, \delta^{(\alpha+s)p}) \subset \widetilde W^{s,p}(\Omega, d^{\alpha p}, d^{(\alpha+s)p})$. For the opposite embedding write
\begin{align}
\label{uno}
|f|^p_{ W^{s,p}(\Omega,  \delta^{(\alpha+s)p})} &=  |f|^p_{\widetilde W^{s,p}(\Omega, \delta^{(\alpha+s)p})} + \int_\O\int_{|x-y|\ge \frac{d(x)}{2}} \frac{|f(x)-f(y)|^p}{|x-y|^{n+sp}} \delta(x,y)^{(\alpha+s)p} \, dy \, dx \nonumber \\
&\le |f|^p_{\widetilde W^{s,p}(\Omega, \delta^{(\alpha+s)p})} + 2^p \int_\O\int_{|x-y|\ge\frac{d(x)}{2}} \frac{|f(x)|^p+ |f(y)|^p}{|x-y|^{n+sp}} \delta(x,y)^{(\alpha+s)p} \, dy \, dx
\end{align}

Observe that
\begin{align}
\label{dos}
\int_\O\int_{|x-y|\ge\frac{d(x)}{2}} \frac{|f(x)|^p}{|x-y|^{n+sp}} \delta(x,y)^{(\alpha+s)p} \, dy \, dx &\le \int_\O |f(x)|^p d(x)^{(\alpha+s)p} \left\{\int_{|x-y|\ge \frac{d(x)}{2}} \frac{1}{|x-y|^{n+sp}} dy\right\} dx \nonumber\\
&\le C \int_\O |f(x)|^p d(x)^{\alpha p} \, dx
\end{align}

Also, when $|x-y|\ge \frac{d(x)}{2}$ we have $d(y)\le |x-y|+d(x)\le 3 |x-y|$ and so 
\begin{align}
\label{tres}
\int_\O\int_{|x-y|\ge\frac{d(x)}{2}} \frac{|f(y)|^p}{|x-y|^{n+sp}} \delta(x,y)^{(\alpha+s)p} \, dy \, dx &\le \int_\O |f(y)|^p d(y)^{(\alpha+s)p} \left\{\int_{|x-y|\ge \frac {d(y)}{3}} \frac{1}{|x-y|^{n+sp}} dx\right\} dy \nonumber \\
&\le C \int_\O |f(y)|^p d(y)^{\alpha p} \, dy
\end{align}

Plugging \eqref{dos} and \eqref{tres} into \eqref{uno} we obtain the desired result. 

\end{proof}

\begin{remark}
\label{remark22}
It is easily seen that the previous proof can be modified replacing $\frac{1}{2}d(x)$ by $\tau d(x)$ for any $0<\tau<1$ in the definition of the $\widetilde W^{s,p}(\Omega, d^{\alpha p}, d^{(\alpha+s)p})$ seminorm, thus proving that the norms of the associated spaces are equivalent. 
\end{remark}

Below  we  recall the essential definitions of the real interpolation method. We refer the reader to \cite[Chapter 1]{L} for more details. 

\begin{defi}
For $0<s<1$ the real interpolation space between $L^p(\Omega, d^{\alpha p})$ and $W^{1,p}(\Omega, d^{\alpha p}, d^{(\alpha+1)p})$ is given by
$$
(L^p(\Omega, d^{\alpha p}), W^{1,p}(\Omega, d^{\alpha p}, d^{(\alpha+1)p}))_{s,p} = \{ f\in L^p(\Omega, d^{\alpha p}) : \|f\|_{(L^p(\Omega, d^{\alpha p}), W^{1,p}(\Omega, d^{\alpha p}, d^{(\alpha+1)p}))_{s,p}} <\infty\}
$$
where 
$$
\|f\|_{(L^p(\Omega, d^{\alpha p}), W^{1,p}(\Omega, d^{\alpha p}, d^{(\alpha+1)p}))_{s,p}} = \left( \int_0^\infty \lambda^{-sp} K(\lambda, f)^p \,  \frac{d\lambda}{\lambda} \right)^\frac{1}{p}
$$
and the $K$ functional is given by 
$$
K(\lambda, f)= \inf\{ \|g\|_{L^p(\Omega, d^{\alpha p})}+ \lambda \|h\|_{W^{1,p}(\Omega, d^{\alpha p}, d^{(\alpha+1)p})} : f= g+h\}
$$
\end{defi}

\section{Proof of the embedding $(L^{\lowercase{p}}(\Omega), W^{\lowercase{1,p}}(\Omega, 1, \lowercase{d^p}))_{\lowercase{s,p}}\subset \widetilde W^{\lowercase{s,p}}(\Omega, 1, \lowercase{d^{sp}})$ }
\label{sec:1}

We begin this section with two necessary lemmas. 

\begin{lemma}
\label{lema-phi}
Let $\phi\in L^1(\O)$, $\phi\ge 0$, $w\in\R^n$ such that $|w|<\frac12$, and $0\le t\le 1$. Then,
\begin{equation}
\label{eq:cambio de variables}
\int_\O \phi(x+t d(x) w) \, dx \le 2^n \int_\O \phi(x) \, dx.
\end{equation}
\end{lemma}
\begin{proof}
Consider the change of variables $y=F(x):= x+t d(x) w$. First of all, observe
that $F(\O)\subset\O$ and that $F$ is injective, since $F(x)=F(\bar x)$ implies that 
$$
|x-\bar x|=t |w| |d(\bar x)-d(x)|\le \frac12 |x-\bar x|
$$
and therefore $x=\bar x$. Consequently,
\begin{equation}
\label{cambio}
\int_\O \phi(x+t d(x) w) \, dx = \int_{F(\O)} \phi(y) |J^{-1}(y)| dy
\end{equation}
where $J=\det DF$. But $DF=I-tB$ with $B=-w (\nabla d)^T$, and
$$
\|B\|=\max_{|v|=1} |Bv|=\max_{|v|=1} |w(\nabla d)^Tv|
\le \max_{|v|=1} |w| |(\nabla d)^Tv|\le |w| |\nabla d| \le \frac12
$$
Then, we have that
$$
|DF^{-1}|=|(I-tB)^{-1}|=\Big|\sum_{j=0}^\infty t^j B^j\Big|
\le 2
$$
and so $|J^{-1}|\le 2^n$, which together with \eqref{cambio} concludes  the proof.	
\end{proof}
\begin{lemma}
\label{rem:suave} 
Let $\Omega \subset \mathbb{R}^n$ be a bounded domain,  $1\le p <\infty$, and $\alpha \ge 0$. Then $C^\infty(\Omega) \cap W^{1,p}(\Omega, d^{\alpha p}, d^{(\alpha+1)p})$ is dense in $ W^{1,p}(\Omega, d^{\alpha p}, d^{(\alpha+1)p})$.
\end{lemma}
\begin{proof}
Since $ W^{1,p}(\Omega, d^{\alpha p}, d^{(\alpha+1)p}) \subset W^{1,p}_{loc}(\Omega)$, by  \cite[Theorem 1]{H} for any $f\in W^{1,p}(\Omega, d^{\alpha p}, d^{(\alpha+1)p})$ and $\varepsilon>0$ there exists  $g \in C^\infty(\Omega)$ such that $f-g\in W^{1,p}_{0}(\Omega)$ and  $\|f-g\|_{W^{1,p}(\Omega)}<\varepsilon$. Hence, 
$\|f-g\|_{W^{1,p}(\Omega, d^{\alpha p}, d^{(\alpha+1)p})}\le C \|f-g\|_{W^{1,p}(\Omega)}< C\varepsilon$ and $g \in W^{1,p}(\Omega, d^{\alpha p}, d^{(\alpha+1)p})$ by the triangle inequality.
\end{proof}

Now we are ready to show that $(L^p(\Omega), W^{1,p}(\Omega, 1, d^p))_{s,p}\subset \widetilde W^{s,p}(\Omega, 1, d^{sp})$. Observe that, since one always has the continuous embedding $(X,Y)_{s,p} \subset X+Y$ and $L^p(\Omega)+ W^{1,p}(\Omega, 1, d^{p})= L^p(\Omega)$, it follows that $\|f\|_{L^p}< C \|f\|_{(L^p(\Omega), W^{1,p}(\Omega, 1, d^p))_{s,p}}$. Hence, it is enough to prove that 
\begin{equation}
\label{interpolado contenido en fraccionario}
|f|_{\widetilde W^{s,p}(\Omega, d^{sp})}^p
\le C \int_0^\infty r^{-sp} K(r, f)^p \frac{dr}{r}.
\end{equation}

We begin by making the change of variables $w=\frac{y-x}{d(x)}$ in  definition \eqref{eq:seminorma_frac} (for $\beta= {sp}$) to obtain
\begin{equation}
\label{seminorma con cambio de variable}
|f|_{\widetilde W^{s,p}(\Omega, d^{sp})}^p
=\int_\Omega \int_{|w|<\frac12} \frac{|f(x+d(x)w)-f(x)|^p}{|w|^{n+sp}} \,  dw  \,\,dx.
\end{equation}
Now, for each $w$, calling $r=|w|$, we consider a decomposition
$f=g_r+ h_r$ such that 
$$
\|g_r\|_{L^p(\Omega)} + r \|\nabla h_r\|_{L^p(\Omega, d^{sp})} \le 2K(r,f). 
$$
By  Lemma \ref{rem:suave} we  may also assume that $h_r$ is smooth.

Since  $|f|_{\widetilde W^{s,p}(\Omega, d^{sp})} \le |g_r|_{\widetilde W^{s,p}(\Omega, d^{sp})} + |h_r|_{\widetilde W^{s,p}(\Omega, d^{sp})}$ it
is enough to estimate each term separately. 
Using \eqref{seminorma con cambio de variable} for $g_r$ we have

$$
\begin{aligned}
|g_r|^p_{\widetilde W^{s,p}(\Omega, d^{sp})}
&= \int_\O  \int_{|w|< \frac12} \frac{|g_r(x+d(x)w) -g_r(x)|^p}{|w|^{n+sp}}  \, dw \, dx \\
	&\le C  \int_{|w|< \frac12} \int_\Omega \frac{|g_r(x)|^p}{|w|^{n+sp}}  \, dx \, dw 
	+ C \int_{|w|<\frac12}\int_{\O}  \frac{|g_r(x+d(x)w)|^p}{|w|^{n+sp}} \, dx \, dw\\
	&\le  C \int_{|w|<\frac12} \frac{\|g_r\|_p^p}{|w|^{n+sp}} \, dw 
\end{aligned}
$$
where in the last step we have used Lemma \ref{eq:cambio de variables} for $\phi=|g_r|^p$ and $t=1$.

On the other hand, observe that for $x\in\O$, $|w|<1/2$ and $0\le t\le 1$
we have $d(x)\le 2 d(x+td(x)w)$, and then
\begin{align*}
|h_r(x+d(x)w) -h_r(x)|^p 
&= \left| \int_0^1 \nabla h_r(x+td(x)w) \cdot d(x) w \, dt \right|^p\\
&\le  \int_0^1 \left|\nabla h_r(x+td(x)w)\right|^p  d(x)^p |w|^p \, dt\\
&\le  2^p\int_0^1 \left|\nabla h_r(x+td(x)w)\right|^p  d(x+td(x)w)^p |w|^p \, dt.
\end{align*}

Therefore,
$$
\begin{aligned}
 |h_r&|^p_{\widetilde W^{s,p}(\Omega, d^{sp})}=\int_{\O} \int_{|w|< \frac12}  \frac{|h_r(x+d(x)w) -h_r(x)|^p}{|w|^{n+sp}}  \, dw \, dx \\
&\le 	
2^p \int_{|w|< \frac12} \frac{|w|^p}{|w|^{n+sp}} \int_0^1 
\int_{\O}\left|\nabla h_r(x+td(x)w)\right|^p  d(x+td(x)w)^p \, dx  \, dt  \, dw 
\\
&\le C \int_{|w|< \frac12} \|\nabla h_r \, d\|^p_{L^p(\Omega)}  \, |w|^{-n-sp+p} \, dw,
 \end{aligned}
$$
where we have used again Lemma \ref{lema-phi}, now for
$\phi=|\nabla h_r|^pd^p$.

Putting together the previous estimates and integrating in polar coordinates
we have

\begin{align*}
	|f|_{\widetilde W^{s,p}(\Omega)}^p &\le C \left( \int_0^\frac12 \frac{r^{n-1}}{r^{n+sp}} \|g_r\|^p_{L^p(\Omega)} \, dr 
	+  \int_0^\frac12 \frac{r^{n-1}}{r^{n+sp-p}} \|\nabla h_r\|^p_{L^p(\Omega, d^p)} \, dr \right)\\
	&\le C \int_0^\frac12 r^{-sp} \Big(\|g_r\|_{L^p(\Omega)} + r \|\nabla h_r\|_{L^p(\Omega, d^{p})}\Big)^p  \, \frac{dr}{r}\\
	&\le C \int_0^\frac12 r^{-sp} K(r,f)^p \,  \frac{dr}{r} 
\end{align*}
which proves \eqref{interpolado contenido en fraccionario} and
 the claimed embedding.

\section{Proof of the embedding $\widetilde W^{\lowercase{s,p}}(\Omega, 1,  \lowercase{d^{sp}}) \subset (L^{\lowercase{p}}(\Omega), W^{\lowercase{1,p}}(\Omega, 1, \lowercase{d^p}))_{\lowercase{s,p}}$ }
\label{sec:2}

Observe first that, since $K(\lambda, f) \le \|f\|_{L^p(\Omega)}$ and $W^{1,p}(\Omega, 1, d^{\alpha  p}) \subset L^p(\Omega)$, we always have
$$
\int_1^\infty (\lambda^{-s} K(\lambda, f))^p \, \frac{d\lambda}{\lambda} \le \|f\|^p_{L^p(\Omega)} \int_1^\infty \lambda^{-sp} \, \frac{d\lambda}{\lambda} \le C \|f\|^p_{L^p(\Omega)}.
$$

Also,  for a given decomposition, 
\begin{align*}
\int_0^1 (\lambda^{1-s} \|h\|_{L^p(\Omega)})^p \, \frac{d\lambda}{\lambda} & \le C \left\{ \int_0^1 (\lambda^{1-s} \|f\|_{L^p(\Omega)} )^p \frac{d\lambda}{\lambda} + \int_0^1 (\lambda^{1-s} \|g\|_{L^p(\Omega)})^p \frac{d\lambda}{\lambda}\right\}\\
&\le C  \left\{ \|f\|_{L^p(\Omega)}^p +  \int_0^1 (\lambda^{-s} \|g\|_{L^p(\Omega)})^p \frac{d\lambda}{\lambda}\right\}.
\end{align*}

Therefore,
\begin{equation}
\label{simplificada} \int_0^\infty (\lambda^{-s} K(\lambda, f))^p \, \frac{d\lambda}{\lambda}  \le C \left\{ \|f\|^p_{L^p(\Omega)} +\int_0^1 \lambda^{-sp} (\|g\|_{L^p(\Omega)} + \lambda \|\nabla h\|_{L^p(\Omega, d^{p})})^p \frac{d\lambda}{\lambda} \right\},
\end{equation}
and to prove the claimed embedding it suffices to bound the integral on the right-hand side.

For this purpose, we will make use of the Whitney decomposition of $\Omega$, whose definition we recall below  (see for example \cite[Chapter VI]{S} for a proof of its existence). 
Given a cube $Q\subset \R^n$, we denote   $Q^*$  the cube with the same center but expanded by a factor $9/8$.  $d(Q,\partial\O)$ denotes
the distance of $Q$ to the boundary of $\Omega$, while $\mbox{diam}(Q)$ and $\ell_Q$ are the
diameter and length of the edges of  $Q$, respectively. 

\begin{defi}
\label{whitneydefi}
Given $\Omega\subset \R^n$ an open bounded set, a  Whitney decomposition
of $\Omega$ is a family $\mathcal{W}$ of closed dyadic cubes
with pairwise disjoint interiors and satisfying the following properties: 
\medskip
\begin{enumerate}
\item [1)]$\Omega=\cup_{Q\in W}Q$
\medskip
\item [2)]$\mbox{diam} (Q)\le d(Q,\partial\Omega)\le 4\, \d(Q) \quad \forall Q \in \mathcal W$
\medskip
\item [3)]$\frac14\d(Q)\le \d(\tilde Q)\le 4 \, \d(Q)$
\quad $\forall Q,\tilde Q \in \mathcal W$ \quad
\mbox{such that}\ \ $Q\cap \tilde Q\neq \emptyset .$
\end{enumerate}
\end{defi}

Let $\mathcal{W}=\{Q\}$ be a Whitney decomposition of $\Omega$.   For a fixed  $0<\lambda\le 1$ we denote by $\mathcal{W}^\lambda=\{Q^\lambda\}$ a new dyadic  decomposition obtained from $\mathcal{W}$ by dividing each $Q \in \mathcal{W}$  in such a way that $\frac12 \lambda \ell_Q\le \ell_{Q^\lambda}\le  \lambda \ell_Q$. Notice that, in particular, this means that $\frac12 \lambda \mbox{diam} (Q)\le  \mbox{diam}({Q^\lambda}) \le  \lambda \mbox{diam} (Q)$.  The center of  $Q^\lambda_j$ in this new partition  will be denoted by $x_j^\lambda$,   and we will write $\ell^\lambda_j$ instead of $\ell_{Q^\lambda_j}$. 

For each $\mathcal{W}^\lambda$ we can define the covering of expanded cubes  ${\mathcal{W}^\lambda}^*=\{(Q_j^\lambda)^*\}$. Observe that it satisfies  $\sum_j \chi_{(Q_j^\lambda)^*}(x) \le C$ for every $x\in \Omega$, and that for $x \in (Q_j^\lambda)^*$, $\frac34 \frac{\mbox{diam}(Q_j^\lambda)}{\lambda} \le d(x) \le \frac{41}4 \frac{\mbox{diam}(Q_j^\lambda)}{\lambda}$. Associated to this covering we can also introduce a smooth partition of unity  $\{\psi_j^\lambda\}$ such that  $\mbox{supp} (\psi_j^\lambda)\subset (Q_j^\lambda)^*$, $0\le \psi_j^\lambda\le 1$,     $\sum_j \psi_j^\lambda = 1$ in $\Omega$, and $\|\nabla \psi_j^\lambda\|_\infty \le \frac{C}{\ell_j^\lambda}$.

For a given (fixed)  $C^\infty$ function  $\varphi \ge 0$ such that $\mbox{supp}(\varphi) \subset B(0,\frac14)$ and $\int \varphi =1$, and for each $t>0$, we define $\varphi_t=t^{-n} \varphi(t^{-1}x)$. Then, for a given $f \in \widetilde W^{s,p}(\Omega, 1, d^{sp})$ we can  define
 $$h^\lambda(y)=\sum_j f_j^\lambda \psi_j^\lambda(y),$$ 
 with  
 $$f_j^\lambda = \int_{\mathbb{R}^n} f* \varphi_{\ell^\lambda_j}(z) \varphi_{\ell^\lambda_j} (z-x_j^\lambda) \, dz,$$
 which is a smooth approximation of $f$.

Now, we are ready to show that 
$$
\int_0^1 \lambda^{-sp} (\|f-h^\lambda\|_{L^p(\Omega)} +\lambda \|\nabla h^\lambda\|_{L^p(\Omega, d^{p})} )^p \frac{d\lambda}{\lambda} \le C |f|_{\widetilde W^{s,p}(\Omega, d^{sp})}^p,
$$
which will prove the embedding $ \widetilde W^{s,p}(\Omega, 1, d^{sp}) \subset (L^p(\Omega), W^{1,p}(\Omega, 1, d^p))_{s,p}$. 

Since the elements of ${\mathcal{W}^\lambda}^*$ have finite overlapping, we have that
\begin{equation}
\label{h-lambda}\|f-h^\lambda\|_{L^p(\Omega)}^p \le  C \sum_j \|f-f_j^\lambda\|_{L^p((Q_j^\lambda)^*)}^p.
\end{equation}
Now, if $y\in (Q_j^\lambda)^*$, 
$$
f(y)- f_j^\lambda = \int_{\mathbb{R}^n} (f(y) - f* \varphi_{\ell_j^\lambda}(z)) \varphi_{\ell_j^\lambda} (z-x_j^\lambda) \, dz,
$$
so, if we let $u(x,t)=f* \varphi_t (x)$ and $g(t)=u(tz+(1-t)y, t\ell^\lambda_j)$,  we have that  
\begin{align*}
f(y)- f* \varphi_{\ell_j^\lambda}(z) &= g(0) - g(1) = - \int_0^1 g'(t) \, dt \\
&= - \int_0^1 \Big[\nabla u (y+t(z-y), t\ell_j^\lambda) \cdot (z-y) + \frac{\partial u}{\partial t}(y+t(z-y), t\ell_j^\lambda) \, \ell_j^\lambda \Big] \, dt.
\end{align*}
Hence,
\begin{align*}f(y)- f_j^\lambda &= \int_{\mathbb{R}^n} (f(y) - f* \varphi_{\ell_j^\lambda}(z)) \varphi_{\ell_j^\lambda} (z-x_j^\lambda) \, dz\\
&= - \int_{\mathbb{R}^n} \int_0^1 \Big[\nabla u (y+t(z-y), t\ell_j^\lambda)\cdot(z-y) + \frac{\partial u}{\partial t}(y+t(z-y), t\ell_j^\lambda) \, \ell_j^\lambda \Big] \, dt \,  \varphi_{\ell_j^\lambda} (z-x_j^\lambda) \, dz.
\end{align*}
Observe that  the integral vanishes unless $z-x_j^\lambda \in \mbox{supp} (\varphi_{\ell_j^\lambda})$, that is $|z-x_j^\lambda|\le \frac14 \ell_j^\lambda$, which means that $z\in Q_j^\lambda$. On the other hand, $y\in (Q_j^\lambda)^*$ and then  $|z-y|\le  9/8\sqrt{n}\ell_j^\lambda$.  Changing variables $x=y + t(z-y)$,  we get  $|x-y|\le 9/8\sqrt{n} t \ell_j^\lambda $
 and also $x\in (Q_j^\lambda)^*$, since $x$ belongs to the segment joining $y$ with $z$. In particular, we can write
\begin{align*}f(y)- f_j^\lambda = -  \int_0^1 \int_{|x-y|<C t \ell_j^\lambda} \Big[\nabla u (x, t\ell_j^\lambda)\cdot\Big(\frac{x-y}{t}\Big) + \frac{\partial u}{\partial t}(x, t\ell_j^\lambda) \, \ell_j^\lambda \Big]  \varphi_{\ell_j^\lambda} \Big(\frac{x-y}{t}+ y-x_j^\lambda\Big) \frac{\chi_{(Q_j^\lambda)^*}(x)}{t^n}  \, dx \, dt.
\end{align*}
Since $\int \nabla \varphi=0$, we have
\begin{align*}
\nabla u(x, t\ell_j^\lambda) &= f*\nabla \varphi_{t\ell_j^\lambda}(x)= \int_{\mathbb{R}^n} f(w) \nabla\varphi_{t\ell_j^\lambda}(x-w) \, dw\\
&= -\int_{\mathbb{R}^n} (f(x)-f(w))  \frac{1}{(t\ell_j^\lambda)^{n+1}} \nabla\varphi \Big(\frac{x-w}{t\ell_j^\lambda}\Big) \, dw
\end{align*}
and, similarly, using that  $\int \frac{\partial \varphi_t}{\partial t}=0$,
\begin{align*}
\frac{\partial u}{\partial t}(x, t\ell_j^\lambda) &= f*\frac{\partial \varphi_{t\ell_j^\lambda}}{\partial t}(x) = \int_{\mathbb{R}^n} f(w)  \frac{\partial \varphi_{t\ell_j^\lambda}}{\partial t}(x-w) \, dw\\
&=  -\int_{\mathbb{R}^n} (f(x)-f(w)) \Big[ -\nabla\varphi \Big(\frac{x-w}{t\ell_j^\lambda}\Big) \cdot \Big(\frac{x-w}{(\ell_j^\lambda)^{n+1} t^{n+2}}\Big)- n \varphi\Big( \frac{x-w}{t\ell_j^\lambda}\Big) \frac{1}{(\ell_j^\lambda)^n t^{n+1}}\Big)\Big] \, dw.
\end{align*}

Therefore we arrive at $f(y)-f_j^\lambda = I_1 + I_2 + I_3$, with 
$$
I_1 =  \int_0^1  \int_{|x-y|<C t \ell_j^\lambda} \int_{\mathbb{R}^n} (f(x)-f(w)) \frac{1}{(t\ell_j^\lambda)^{n+1}} \nabla\varphi \Big(\frac{x-w}{t\ell_j^\lambda}\Big) \cdot\Big(\frac{x-y}{t}\Big)   \varphi_{\ell_j^\lambda}\Big(\frac{x-y}{t}+ y-x_j^\lambda\Big) \frac{\chi_{(Q_j^\lambda)^*}(x)}{t^n} \, dw \, dx \, dt,
$$
$$
I_2 =  \int_0^1  \int_{|x-y|<C t \ell_j^\lambda} \int_{\mathbb{R}^n} (f(x)-f(w)) \frac{1}{(t\ell_j^\lambda)^{n+1}} \nabla\varphi \Big(\frac{x-w}{t\ell_j^\lambda}\Big) \cdot\Big(\frac{x-w}{t}\Big)   \varphi_{\ell_j^\lambda}\Big(\frac{x-y}{t}+ y-x_j^\lambda\Big) \frac{\chi_{(Q_j^\lambda)^*}(x)}{t^n} \, dw \, dx  \, dt,
$$
$$
I_3 =   \int_0^1  \int_{|x-y|<C t \ell_j^\lambda} \int_{\mathbb{R}^n} (f(x)-f(w)) \frac{\ell_j^\lambda}{(t\ell_j^\lambda)^{n+1}} \varphi \Big(\frac{x-w}{t\ell_j^\lambda}\Big)    \varphi_{\ell_j^\lambda}\Big(\frac{x-y}{t}+ y-x_j^\lambda\Big) \frac{n \chi_{(Q_j^\lambda)^*}(x)}{t^n} \, dw  \, dx \, dt. $$
Using that $\frac{x-w}{t\ell_j^\lambda}\in \mbox{supp} (\varphi)$, we have that $|x-w|< \frac14 t\ell_j^\lambda$ and  we can bound
\begin{align*}
|f(y)-f_j^\lambda| &\le C \int_0^1  \int_{|x-y|<C t \ell_j^\lambda} \int_{\mathbb{R}^n} |f(x)-f(w)| \frac{\ell_j^\lambda}{(t\ell_j^\lambda)^{n+1}} \left|\nabla\varphi \Big(\frac{x-w}{t\ell_j^\lambda}\Big)\right| \left| \varphi_{\ell_j^\lambda}\Big(\frac{x-y}{t}+ y-x_j^\lambda\Big)\right| \frac{\chi_{(Q_j^\lambda)^*}(x)}{t^n} \, dw  \, dx \, dt \\
&+ C \int_0^1  \int_{|x-y|<C t \ell_j^\lambda} \int_{\mathbb{R}^n} |f(x)-f(w)| \frac{\ell_j^\lambda}{(t\ell_j^\lambda)^{n+1}} \left|\varphi \Big(\frac{x-w}{t\ell_j^\lambda}\Big)\right| \left| \varphi_{\ell_j^\lambda}\Big(\frac{x-y}{t}+ y-x_j^\lambda\Big)\right| \frac{\chi_{(Q_j^\lambda)^*}(x)}{t^n} \, dw  \, dx \, dt 
\end{align*}
Moreover, recalling that for $x\in (Q_j^\lambda)^*$,  $d(x)\ge  \frac{3}{4} \frac{\ell_j^\lambda}{\lambda}$, we also have
$$
|x-w|< \frac14 t \ell_j^\lambda \le \frac14 \frac{\ell_j^\lambda}{\lambda} \le \frac13 d(x) < \frac{d(x)}{2},
$$
which means 
\begin{align*}
&|f(y)-f_j^\lambda| \\
&\le C \int_0^1  \int_{|x-y|<C t \ell_j^\lambda}  \int_{|x-w|<\frac{d(x)}{2}} |f(x)-f(w)| \frac{\ell_j^\lambda}{(t\ell_j^\lambda)^{n+1}} \left|\nabla\varphi \Big(\frac{x-w}{t\ell_j^\lambda}\Big)\right| \left| \varphi_{\ell_j^\lambda}\Big(\frac{x-y}{t}+ y-x_j^\lambda\Big)\right| \frac{\chi_{(Q_j^\lambda)^*}(x)}{t^n} \, dw  \, dx \, dt \\ 
&+ C \int_0^1  \int_{|x-y|<C t \ell_j^\lambda}  \int_{|x-w|<\frac{d(x)}{2}} |f(x)-f(w)| \frac{\ell_j^\lambda}{(t\ell_j^\lambda)^{n+1}} \left|\varphi \Big(\frac{x-w}{t\ell_j^\lambda}\Big)\right| \left| \varphi_{\ell_j^\lambda}\Big(\frac{x-y}{t}+ y-x_j^\lambda\Big)\right| \frac{\chi_{(Q_j^\lambda)^*}(x)}{t^n} \, dw  \, dx \, dt\\
&= I + II.
\end{align*}

We begin by bounding $I$.  H\"older's inequality in $dx \, dt$ gives us
\begin{align*}
I &\le C \left(  \int_0^1   \int_{|x-y|<C t \ell_j^\lambda} \Big( \int_{|x-w|<\frac{d(x)}{2}} |f(x)-f(w)|  \frac{1}{(\ell_j^\lambda)^n} \frac{\chi_{(Q_j^\lambda)^*}(x)}{t^{\frac{n+1}{p}+\frac{\varepsilon}{p'}+n}}\left|\nabla\varphi  \Big(\frac{x-w}{t\ell_j^\lambda}\Big)\right|  \, dw\Big)^p dx \, dt \right)^\frac{1}{p}  (\ell^\lambda_j)^{-\frac{n}{p}},\\
\end{align*}
where we have used that, for $\varepsilon>0$ to be chosen below,
 \begin{align*}
\Big( \int_0^1&   \int_{\mathbb{R}^n}\frac{1}{t^{n+1-\varepsilon}}  \Big| \varphi_{\ell_j^\lambda}\Big(\frac{x-y}{t}+ y-x_j^\lambda\Big)\Big|^{p'} dt \, dx\Big)^\frac{1}{p'}  \le C (\ell_j^\lambda)^{-\frac{n}{p}}.
\end{align*}

A further  application of H\"older's inequality in $dw$ leads to
$$
I \le C \left(  \int_0^1  \int_{|x-y|<C t \ell_j^\lambda} \int_{|x-w|<\frac{d(x)}{2}} |f(x)-f(w)|^p \chi_{|x-w|<\frac14 t\ell_j^\lambda} \, dw \,   \frac{\chi_{(Q_j^\lambda)^*}(x)}{t^{2n+1+\varepsilon(p-1)}} \, dx \, dt \right)^\frac{1}{p}  (\ell_j^\lambda)^{-\frac{n}{p}+ \frac{n}{p'}-n }.\\
$$
where we have used that
$$
\left(\int_{\mathbb{R}^n} \Big| \nabla\varphi\Big( \frac{x-w}{t\ell_j^\lambda}\Big)\Big|^{p'} \, dw\right)^\frac{1}{p'} \le C (t\ell_j^\lambda)^\frac{n}{p'}.
$$

The term  $II$ can be bounded similarly  observing that 
$$
\left(\int_{\mathbb{R}^n} \Big| \varphi\Big( \frac{x-w}{t\ell_j^\lambda}\Big)\Big|^{p'} \, dw\right)^\frac{1}{p'} \le C (t\ell_j^\lambda)^\frac{n}{p'}.
$$
Hence, for $y\in (Q_j^\lambda)^*$, we arrive at 
$$ 
|f(y)-f_j^\lambda|\le C \left(  \int_0^1  \int_{|x-y|<C t \ell_j^\lambda}  \int_{|x-w|<\frac{d(x)}{2}} |f(x)-f(w)|^p \chi_{|x-w|<\frac14 t\ell_j^\lambda } \, dw \,  \frac{\chi_{(Q_j^\lambda)^*}(x)}{t^{2n+1+\varepsilon(p-1)}} \, dx \, dt \right)^\frac{1}{p}  (\ell_j^\lambda)^{-\frac{n}{p}+ \frac{n}{p'}-n }.$$

Replacing this expression in \eqref{h-lambda}, we get
\begin{align*}
\int_0^1 &\lambda^{-sp} \|f-h^\lambda\|_{L^p(\Omega)}^p \frac{d\lambda}{\lambda}\\
&\le C \int_0^1 \sum_j \int_{(Q_j^\lambda)^*} \int_0^1   \int_{|x-y|<C t \ell_j^\lambda} \int_{|x-w|<\frac{d(x)}{2}} |f(x)-f(w)|^p \chi_{|x-w|<\frac14 t\ell_j^\lambda} \, dw \,  \frac{(\ell_j^\lambda)^{-2n} \lambda^{-sp} \chi_{(Q_j^\lambda)^*}(x)}{t^{2n+1+\varepsilon(p-1)}} \, dx \, dt  \, dy \,  \frac{d\lambda}{\lambda}\\
&\le C \int_0^1 \sum_j  \int_{(Q_j^\lambda)^*} \int_0^1  \int_{|x-w|<\frac{d(x)}{2}} |f(x)-f(w)|^p \chi_{|x-w|<\frac14 t\ell_j^\lambda} \, dw \int_{|x-y|< Ct\ell_j^\lambda} dy \,  \frac{ (\ell_j^\lambda)^{-2n} \lambda^{-sp}}{t^{2n+1+\varepsilon(p-1)}} \, dt \, dx  \, \frac{d\lambda}{\lambda}\\
&\le C \int_0^1 \sum_j  \int_{(Q_j^\lambda)^*} \int_0^1  \int_{|x-w|<\frac{d(x)}{2}} |f(x)-f(w)|^p \chi_{|x-w|<\frac14 t\ell_j^\lambda} \, dw \,  \frac{ (\ell_j^\lambda)^{-n} \lambda^{-sp}}{t^{n+1+\varepsilon(p-1)}} \, dt \, dx  \, \frac{d\lambda}{\lambda}\\
&\le C \int_0^1 \sum_j  \int_{(Q_j^\lambda)^*} \int_0^1  \int_{|x-w|<\frac{d(x)}{2}} |f(x)-f(w)|^p \chi_{|x-w|<\frac14 t \lambda d(x)} \, dw \,  \frac{ d(x)^{-n} \lambda^{-n-sp}}{t^{n+1+\varepsilon(p-1)}} \, dt \, dx  \, \frac{d\lambda}{\lambda},
\end{align*}
where in the last line we have used that  for $x\in (Q_j^\lambda)^*$,  $\ell_j^\lambda $ is equivalent to $\lambda d(x)$. Now we add up in $j$ and change the order of integration to obtain
\begin{align*}
\int_0^1 &\lambda^{-sp} \|f-h^\lambda\|_{L^p(\Omega)}^p \frac{d\lambda}{\lambda}\\
&\le C \int_0^1   \int_{\Omega} \int_0^1  \int_{|x-w|<\frac{d(x)}{2}} |f(x)-f(w)|^p \chi_{|x-w|<\frac14 t \lambda d(x)} \, dw \,  \frac{ d(x)^{-n} \lambda^{-n-sp}}{t^{n+1+\varepsilon(p-1)}} \, dt \, dx  \, \frac{d\lambda}{\lambda}\\
&\le C    \int_{\Omega} \int_0^1  \int_{|x-w|<\frac{d(x)}{2}} |f(x)-f(w)|^p \, dw  \int_{ \frac{4 |x-w|}{t d(x)}}^\infty \lambda^{-n-sp-1} \, d\lambda \,  \frac{ d(x)^{-n} }{t^{n+1+\varepsilon(p-1)}} \, dt \, dx  \\
&\le C    \int_{\Omega} \int_0^1  \int_{|x-w|<\frac{d(x)}{2}} \frac{|f(x)-f(w)|^p}{|x-w|^{n+sp}} \, dw  \,  \frac{ d(x)^{sp} }{t^{-sp+1+\varepsilon(p-1)}} \, dt \, dx  \\
&\le C    \int_{\Omega}  \int_{|x-w|<\frac{d(x)}{2}} \frac{|f(x)-f(w)|^p}{|x-w|^{n+sp}} \, dw  \, d(x)^{sp}  \, dx\\
&=|f|^p_{\widetilde W^{s,p}(\Omega, 1, d^{sp})},  
\end{align*}
where in the last inequality we have used that we can choose $\varepsilon>0$  such that $-sp+1+\varepsilon(p-1)<1$.

For the other term of the $K$ functional, we write
$$
|\nabla h^\lambda(y)| = \Big|\sum_j f_j^\lambda \nabla\psi_j^\lambda(y)\Big| \le \sum_j |f_j^\lambda - f(y)| |\nabla\psi_j^\lambda(y)| \le C \sum_j |f_j^\lambda-f(y)| \frac{1}{\ell_j^\lambda} \chi_{(Q_j^\lambda)^*}(y).
$$
Therefore,
\begin{align*}
\int_0^1  \lambda^{p(1-s)} \|\nabla h^\lambda\|_{L^p(\Omega, d^p)}^p \, \frac{d\lambda}{\lambda} &\le C \int_0^1 \sum_j  \lambda^{p(1-s)} \|(f-f_j^\lambda) (\ell_j^\lambda)^{-1}\|_{L^p((Q_j^\lambda)^*, d^p)}^p \, \frac{d\lambda}{\lambda}\\
&\le C \int_0^1 \sum_j  \lambda^{p(1-s)} \|(f-f_j^\lambda) \lambda^{-1}\|_{L^p((Q_j^\lambda)^*)}^p \, \frac{d\lambda}{\lambda}\\
&\le C \int_0^1 \sum_j  \lambda^{-sp} \|f-f_j^\lambda\|_{L^p((Q_j^\lambda)^*)}^p \, \frac{d\lambda}{\lambda}
\end{align*}
and this expression can be bounded as before.

Summing up, by \eqref{simplificada} and the previous bounds, we have that
\begin{align*}
\int_0^\infty \lambda^{-sp} K(\lambda, f)^p \, \frac{d\lambda}{\lambda} &\le C \left\{ \|f\|^p_{L^p(\Omega)} + \int_0^1 \lambda^{-sp} (\|f-h^\lambda\|_{L^p(\Omega)} +\lambda \|\nabla h^\lambda\|_{L^p(\Omega, d^{p})} )^p \frac{d\lambda}{\lambda} \right\} \\
& \le C \left\{ \|f\|^p_{L_p(\Omega)} + |f|^p_{\widetilde W^{s,p}(\Omega, 1, d^{sp})} \right\},
\end{align*}
and, therefore, the claimed embedding follows.
\section{Proof of Theorem 1.1}
\label{sec:3}

As said before, the proof of Theorem \ref{inter_general} for any $\alpha> 0$ is very similar to the proof developed in detail for $\alpha=0$ in the previous sections, so we will only  indicate briefly what are the necessary changes. 

To prove the embedding $(L^p(\Omega, d^{\alpha p}), W^{1,p}(\Omega, d^{\alpha p}, d^{(\alpha+1)p}))_{s,p} \subset \widetilde W^{s,p}(\Omega, d^{\alpha p}, d^{(\alpha+s)p})$ proceed as in Section 3. For the term 
$$
|g_r|_{\widetilde W(\Omega, d^{(\alpha+s)p})} = \int_\Omega \int_{|w|<\frac12} \frac{|f(x+w d(x)) - f(x)|^p}{|w|^{n+sp}} \, dw \, d(x)^{\alpha p} \, dx, 
$$ 
split the integral as before,  observe that $d(x)\le 2 d(x+ wd(x))$ and use Lemma \ref{lema-phi} for $|g_r \, d^{\alpha }|^p$ and $t=1$ to arrive at the desired bound. The part involving 
 $\nabla h_r$ follows exactly as in Section 3 making the necessary changes in the exponents.

To prove the embedding $ \widetilde W^{s,p}(\Omega, , d^{\alpha p}, d^{(\alpha+s)p}) \subset (L^p(\Omega, d^{\alpha p}), W^{1,p}(\Omega, d^{\alpha p}, d^{(\alpha+1)p}))_{s,p}$ , observe that the generalization of \eqref{simplificada} is straightforward and use the same pointwise bound as in Section 4 to arrive at
 \begin{align*}
&\int_0^1 \lambda^{-sp} \|f-h^\lambda\|_{L^p(\Omega, d^{\alpha p})}^p \frac{d\lambda}{\lambda}\\
&\le C \int_0^1 \sum_j \int_{(Q_j^\lambda)^*} \int_0^1   \int_{|x-y|<C t \ell_j^\lambda} \int_{|x-w|<\frac{d(x)}{2}} |f(x)-f(w)|^p \chi_{|x-w|<\frac14 t\ell_j^\lambda} \, dw \,  \frac{(\ell_j^\lambda)^{-2n} \lambda^{-sp}  \chi_{(Q_j^\lambda)^*}(x) d(y)^{\alpha p}}{t^{2n+1+\varepsilon(p-1)}} \, dx \, dt  \, dy \,  \frac{d\lambda}{\lambda}.
\end{align*}
Use that $d(y) \le C d(x)$ to bound $d(y)^{\alpha p}$ by $C d(x)^{\alpha p}$ and follow the rest of the steps in that section to finish the proof.

\end{document}